\documentclass[11pt,leqno]{amsart}
\usepackage{amsmath,amsthm,amssymb}
\usepackage{tabularx}
\usepackage{comment}
\newcommand{\GF}{{\mathbb F}}
\newcommand{\FF}{{\mathbb F}}
\newcommand{\F}{{\mathbb F}}

\newcommand{\R}{{\mathbb R}}
\newcommand{\RR}{{\mathbb R}}

\newcommand{\NN}{{\mathbb N}}

\DeclareMathOperator{\Hom}{Hom}
\DeclareMathOperator{\kernel}{Ker}

\newcommand{\wt}{{\rm wt}}
\newcommand{\supp}{{\rm supp}}

\newcommand{\allone}{\mathbf{1}}

\DeclareMathOperator{\Harm}{Harm}

\newcommand{\nexteq}{\displaybreak[0]\\ &=}

\usepackage{color} 
\usepackage{ulem}

\newtheorem{thm}{Theorem}[section]
\newtheorem{lem}[thm]{Lemma}
\newtheorem{cor}[thm]{Corollary}

\theoremstyle{definition}
\newtheorem{df}[thm]{Definition}

\numberwithin{equation}{section}

\title[Jacobi polynomials and harmonic weight enumerators]
{Jacobi polynomials and harmonic weight enumerators of 
the first-order Reed--Muller codes and the extended Hamming codes}

\author{Tsuyoshi Miezaki*}
\address{		Faculty of Science and Engineering, 
		Waseda University, 
		Tokyo 169--8555, Japan
}
\email{miezaki@waseda.jp} 
\author{Akihiro Munemasa}
\address{Graduate School of Information Sciences, Tohoku University, Sendai
980--8579, Japan}
\email{munemasa@math.is.tohoku.ac.jp}



\begin{document}
\begin{abstract}
In the present paper, 
we give harmonic weight enumerators and Jacobi polynomials for 
the first-order Reed--Muller codes and the extended Hamming codes. 
As a corollary, 
we show the nonexistence of combinatorial $4$-designs in these codes. 
\end{abstract}
\maketitle


\section{Introduction}

Let $m$ be a positive integer, and set $V=\F_2^m$.
The first-order Reed--Muller code $RM(1,m)$ is defined as
the subspace of $\F_2^V$ consisting of affine linear functions:
\[RM(1,m)=\{(\lambda(x)+b)_{x\in V}\mid \lambda\in V^*,\;b\in\F_2\},\]
where $V^*=\Hom(V,\F_2)$. 
We remark that 
the weight enumerator of $RM(1,m)$ is 
\[
x^{2^m}+(2^{m+1}-2)x^{2^{m-1}}y^{2^{m-1}}+y^{2^m}.
\]
It is well known that 
the dual code of $RM(1,m)$ is 
isomorphic to an extended Hamming code 
$H_{2^m}$ \cite{MS}. 

Let $C=RM(1,m)$ or $H_{2^m}$, and 
$C_\ell:=\{c\in C\mid \wt(c)=\ell\}$. 
In this paper, we call $C_\ell$ a shell of the code $C$ 
whenever it is non-empty.
Shells of $RM(1,m)$ and $H_{2^m}$ are known to 
support combinatorial $3$-designs by 
the Assmus--Mattson theorem (see Theorem \ref{thm:assmus-mattson}) 
or the transitivity argument (see \cite[Ch.~13.~\S 9]{MS}). 
More precisely, 
the set $\mathcal{B}(C_\ell):=\{\supp(x)\mid x\in C_\ell\}$
forms the set of blocks of a combinatorial $3$-design. 

In \cite{MN-TEC}, H.~Nakasora and the first named author gave 
the first nontrivial examples of a code 
that supports combinatorial $t$-designs for all weights 
obtained from the Assmus--Mattson theorem and 
that supports $t'$-designs for some weights with some $t'>t$ 
(see also \cite{{BS},{Dillon-Schatz},{MMN},{mn-typeI}}). 
Hence, it is natural to ask whether 
certain shells of $RM(1,m)$ and $H_{2^m}$ support combinatorial $4$-designs. The aim of the present paper is to settle this problem by 
computing the Jacobi polynomials and harmonic weight enumerators for 
the first-order Reed--Muller codes and the extended Hamming codes. 
The definitions of 
the Jacobi polynomials and harmonic weight enumerators 
will be given in Section \ref{sec:pre}. 

\begin{thm}\label{thm:Jacobi}
Let $C=RM(1,m)$ and
$T=\{0,u_1,u_2,u_3\}\in\binom{V}{4}$.
\begin{enumerate}
\item\label{thm:1.1-1}
\begin{enumerate}
\item
If $u_1+u_2\neq u_3$, then
\begin{align*}
J_{C,T}(w,z,x,y)=&w^4x^{2^m-4} + (2^{m-3}-1)w^4x^{2^{m-1}-4}y^{2^{m-1}} \\
&+ 2^{m-1}w^3 z x^{2^{m-1}-3} y^{2^{m-1}-1} 
+ 3\cdot 2^{m-2}w^2z^2x^{2^{m-1}-2}y^{2^{m-1}-2}\\
&+ 2^{m-1}wz^3x^{2^{m-1}-1}y^{2^{m-1}-3}
+ (2^{m-3}-1)z^4x^{2^{m-1}}y^{2^{m-1}-4}\\
&+ z^4y^{2^m-4}. 
\end{align*}
\item
If $u_1+u_2= u_3$, then
\begin{align*}
J_{C,T}(w,z,x,y)=&w^4x^{2^{m}-4} 
+ (2^{m-2}-1)w^4x^{2^{m-1}-4}y^{2^{m-1}} \\
&+ 3\cdot 2^{m-1}w^2z^2x^{2^{m-1}-2}y^{2^{m-1}-2} \\
&+ (2^{m-2}-1)z^4x^{2^{m-1}}y^{2^{m-1}-4} 
+ z^4y^{2^{m}-4}. 
\end{align*}

\end{enumerate}
\item\label{thm:1.1-2}
\begin{enumerate}
\item
If $u_1+u_2\neq u_3$, then
\begin{align*}
J_{C^\perp,T}(w,z,x,y)=&\frac{1}{2^{m+1}}\Big((w+z)^4(x+y)^{2^m-4}\\
&+ (2^{m-3}-1)(w+z)^4(x+y)^{2^{m-1}-4}(x-y)^{2^{m-1}} \\
&+ 2^{m-1}(w+z)^3 (w-z) (x+y)^{2^{m-1}-3} (x-y)^{2^{m-1}-1} \\
&+ 3\cdot 2^{m-2}(w+z)^2(w-z)^2(x+y)^{2^{m-1}-2}(x-y)^{2^{m-1}-2}\\
&+ 2^{m-1}(w+z)(w-z)^3(x+y)^{2^{m-1}-1}y^{2^{m-1}-3}\\
&+ (2^{m-3}-1)(w-z)^4(x+y)^{2^{m-1}}(x-y)^{2^{m-1}-4} \\
&+ (w-z)^4(x-y)^{2^m-4}\Big). 
\end{align*}
\item
If $u_1+u_2= u_3$, then
\begin{align*}
J_{C^\perp,T}(w,z,x,y)=&\frac{1}{2^{m+1}}\Big((w+z)^4(x+y)^{2^{m}-4} \\
&+ (2^{m-2}-1)(w+z)^4(x+y)^{2^{m-1}-4}(x-y)^{2^{m-1}} \\
&+ 3\cdot 2^{m-1}(w+z)^2(w-z)^2(x+y)^{2^{m-1}-2}(x-y)^{2^{m-1}-2} \\
&+ (2^{m-2}-1)(w-z)^4(x+y)^{2^{m-1}}(x-y)^{2^{m-1}-4} \\
&+ (w-z)^4(x-y)^{2^{m}-4}\Big). 
\end{align*}
\end{enumerate}
\end{enumerate}
\end{thm}

\begin{thm}\label{thm:main1}
Let $C=RM(1,m)$ and $U\subset V$ be a three-dimensional subspace of $V$. 
We assume that 
$f$ is a harmonic function of degree $k\in \NN$ such that 
$f=0$ on $\binom{V}{k}\setminus \binom{U}{k}$. 
Then we have 
\begin{align*}
w_{C,f}(x,y)=&2^{m-3}x^{2^{m-1}}y^{2^{m-1}}\sum_{a\in H_8, \wt(a)=4}\widetilde{f}(a),\\
w_{C^\perp,f}(x,y)=&(-1)^k(xy)^k {2^{2^{m-1}-2}}\\
&\left(\frac{x+y}{\sqrt{2}}\right)^{2^{m-1}-k}\left(\frac{x-y}{\sqrt{2}}\right)^{2^{m-1}-k}
\sum_{a\in H_8, \wt(a)=4}\widetilde{f}(a). 
\end{align*}
\end{thm}
We show, as a corollary, the nonexistence of combinatorial $4$-designs in these codes. 
\begin{cor}\label{cor:main}
Let $C=RM(1,m)$ or $H_{2^m}$. 
Then for any $\ell\in \NN$, 
$C_{\ell}$ is not a combinatorial $4$-design. 
\end{cor}

This paper is organized as follows. 
In Section~\ref{sec:pre}, 
we define and give some basic properties of 
codes, 
combinatorial $t$-designs, 
Jacobi polynomials, and harmonic weight enumerators 
used in this paper.
In Sections~\ref{sec:Jacobi}, \ref{sec:harm}, and \ref{sec:cor}, 
we prove Theorems~\ref{thm:Jacobi} and \ref{thm:main1} and 
Corollary \ref{cor:main}, respectively.

All computer calculations reported in this paper were done using 
{\sc Magma}~\cite{Magma} and {\sc Mathematica}~\cite{Mathematica}. 

\section{Preliminaries}\label{sec:pre}

\subsection{Codes and combinatorial $t$-designs}

A binary linear code $C$ of length $n$ is a linear subspace of $\FF_{2}^{n}$. 
An inner product $({x},{y})$ on $\FF_2^n$ is given 
by
\[
(x,y)=\sum_{i=1}^nx_iy_i,
\]
where $x,y\in \FF_2^n$ with $x=(x_1,x_2,\ldots, x_n)$ and 
$y=(y_1,y_2,\ldots, y_n)$. 
The dual of a linear code $C$ is defined as 
\[
C^{\perp}=\{{y}\in \FF_{2}^{n}\mid ({x},{y}) =0 \text{ for all }{x}\in C\}. 
\]
For $x \in\FF_2^n$,
the weight $\wt(x)$ is the number of its nonzero components. 

A combinatorial $t$-design 
is a pair 
$\mathcal{D}=(\Omega,\mathcal{B})$, where $\Omega$ is a set of points of 
cardinality $v$, and $\mathcal{B}$ is a collection of $k$-element subsets
of $\Omega$ called blocks, with the property that any $t$ points are 
contained in precisely $\lambda$ blocks.

The support of a vector 
${x}:=(x_{1}, \dots, x_{n})$, 
$x_{i} \in \GF_{2}$ is 
the set of indices of its nonzero coordinates: 
${\rm supp} ({x}) = \{ i \mid x_{i} \neq 0 \}$\index{$supp (x)$}.
Let $\Omega:=\{1,\ldots,n\}$ and 
$\mathcal{B}(C_\ell):=\{\supp({x})\mid {x}\in C_\ell\}$. 
Then for a code $C$ of length $n$, 
we say that the shell $C_\ell$ is a combinatorial $t$-design if 
$(\Omega,\mathcal{B}(C_\ell))$ is a combinatorial $t$-design. 

The following theorem is from Assmus and Mattson \cite{assmus-mattson}. It is one of the 
most important theorems in coding theory and design theory:
\begin{thm}[\cite{assmus-mattson}] \label{thm:assmus-mattson}

Let $C$ be a linear code of 
length $n$ over $\FF_q$ with minimum weight $d$. 
Let $C^\perp$ denote the dual code of $C$, with 
minimum weight $d^\perp$. 
Suppose that an integer $t$ $($$1 \leq t \leq n$$)$ is 
such that there are at most $d-t$ weights of $C^\perp$ 
in $\{1, 2,\ldots , n - t\}$, 
or such that there are at most $d^\perp - t$ weights of $C$ 
in $\{1, 2, \ldots , n-t\}$. 
Then the supports of the words of any fixed weight 
in $C$ form a $t$-design $($with possibly repeated blocks$)$.
\end{thm}

\subsection{Jacobi polynomials}

Let $C$ be a binary code of length $n$ and $T\subset [n]:=\{1,\ldots,n\}$. 
Then the Jacobi polynomial of $C$ with $T$ is defined as follows \cite{Ozeki}:
\[
J_{C,T} (w,z,x,y) :=\sum_{c\in C}w^{m_0(c)} z^{m_1(c)}x^{n_0(c)}y^{n_1(c)}, 
\]
where for $c=(c_1,\ldots,c_n)$, 
\begin{align*}
m_i(c)&=|\{j\in T \mid c_j=i \}|,\\
n_i(c)&=|\{j\in [n]\setminus T\mid c_j=i \}|.
\end{align*}

The following is a generalization of the classical MacWilliams identity: 
\begin{thm}[\cite{Ozeki}]\label{thm:Mac-Jacobi}
Let $C$ be a binary code of length $n$ and $T\subset [n]$. 
Then we have 
\[
J_{C^\perp,T}(w,z,x,y) =\frac{1}{|C|}J_{C,T}(w + z,w - z,x + y,x - y).
\]
\end{thm}


It is easy to see that 
$C_\ell$ is a combinatorial $t$-design 
if and only if 
the coefficient of $z^{t}x^{n-\ell}y^{\ell-t}$ 
in $J_{C,T}$ is independent of the choice of $T$ with $|T|=t$. 

\subsection{Harmonic weight enumerators}\label{sec:Har}


In this subsection, we review the concept of 
harmonic weight enumerators.

Let $\Omega=\{1, 2,\ldots,n\}$ be a finite set (which will be the set of coordinates of the code) and 
let $X$ be the set of its subsets, while, for each $k= 0,1,\dots, n$, 
let $X_{k}$ be the set of its $k$-subsets.
We denote by $\R X$ and $\R X_k$ the 
real vector spaces spanned by the elements of $X$
and $X_{k}$, respectively.
An element of $\R X_k$ is denoted by
$$f=\sum_{z\in X_k}f(z)z$$
and is identified with the real-valued function on $X_{k}$ given by 
$z \mapsto f(z)$. 

An element $f\in \R X_k$ can be extended to an element $\widetilde{f}\in \R X$ by setting, for all $u \in X$,
$$\widetilde{f}(u)=\sum_{z\in X_k, z\subset u}f(z).$$
The differentiation $\gamma$ is the operator on $\RR X$ defined by linearity from 
$$\gamma(z) =\sum_{y\in X_{k-1},y\subset z}y$$
for all $z\in X_k$ and for all $k=0,1, \ldots, n$. 
An element $f\in \R X_k$ satisfying $\gamma(f)=0$
is called a harmonic function of degree $k$,
and we denote by $\Harm_{k}$ the set of all 
harmonic functions of degree $k$:
$$\Harm_k =\ker(\gamma|_{\R X_k}).$$
\begin{thm}[{{\cite[Theorem 7]{Delsarte}}}]\label{thm:design}
A set $\mathcal{B} \subset X_{m}$ $($where $m \leq n$$)$ of blocks is a $t$-design 
if and only if $\sum_{b\in \mathcal{B}}\widetilde{f}(b)=0$ 
for all $f\in \Harm_k$, $1\leq k\leq t$. 
\end{thm}

In \cite{Bachoc}, the harmonic weight enumerator associated with a binary linear code $C$ was defined as follows. 
\begin{df}
Let $C$ be a binary code of length $n$ and let $f\in\Harm_{k}$. 
The harmonic weight enumerator associated with $C$ and $f$ is
$$w_{C,f}(x,y)=\sum_{{c}\in C}\widetilde{f}({c})x^{n-\wt({c})}y^{\wt({c})},$$
where we write $\widetilde{f}(\supp(c))$ as $\widetilde{f}({c})$ for short.
\end{df}

Bachoc proved the following MacWilliams-type equality. 
\begin{thm}[\cite{Bachoc}] \label{thm: Bachoc iden.} 
Let $C$ be a binary code of length $n$ and $f\in\Harm_{k}$. 
Let 
$w_{C,f}(x,y)$ be 
the harmonic weight enumerator associated with $C$ and $f$. 
Then $$w_{C,f}(x,y)= (xy)^{k} Z_{C,f}(x,y),$$
where $Z_{C,f}$ is a homogeneous polynomial of degree $n-2k$ and satisfies
$$Z_{C^{\bot},f}(x,y)= (-1)^{k} \frac{2^{n/2}}{|C|} Z_{C,f} \left( \frac{x+y}{\sqrt{2}}, \frac{x-y}{\sqrt{2}} \right).$$
\end{thm}
It follows from Theorem \ref{thm:design} that 
$C_\ell$ is a combinatorial $t$-design if and only if 
the coefficient of $x^{n-\ell}y^\ell$ in $w_{C,f}(x,y)$ vanishes. 
\section{Proof of Theorem \ref{thm:Jacobi}}\label{sec:Jacobi}


In this section, 
we give a proof of Theorem \ref{thm:Jacobi}. 
Let $V=\FF_2^m$ and $C=RM(1,m)$.
For $c=(\lambda(x) + b)_{x\in V}$ and $T\subset V$, we define 
\[
c\vert_T:=(\lambda(x) + b)_{x\in T}. 
\]

\begin{lem}\label{lem:1}
Let 
$T=\{0,u_1,u_2,u_3\}\in\binom{V}{4}$.
\begin{enumerate}
\item [(1)]\label{lem1i1}
If $u_1+u_2\neq u_3$, then
\[
|\{c\in C\setminus\{0,\allone\}\mid \wt(c|_T)=i\}|
=\begin{cases}
2^{m-3}-1&\text{if $i=0,4$,}\\
2^{m-1}&\text{if $i=1,3$,}\\
3\cdot 2^{m-2}&\text{if $i=2$.}
\end{cases}\]
\item [(2)]\label{lem1i2}
If $u_1+u_2= u_3$, then
\[
|\{c\in C\setminus\{0,\allone\}\mid \wt(c|_T)=i\}|
=\begin{cases}
2^{m-2}-1&\text{if $i=0,4$,}\\
0&\text{if $i=1,3$,}\\
3\cdot 2^{m-1}&\text{if $i=2$.}
\end{cases}\]
\end{enumerate}
\end{lem}
\begin{proof}
By identifying
$(\lambda(x))_{x\in V}$ with $\lambda\in V^*$, 
we have $C=V^*\cup(V^*+\allone)$.
For $i=0,1,\dots,4$, define
\begin{align*}
a_i&=|\{c\in C\setminus\{0,\allone\}\mid \wt(c|_T)=i\}|,\\
b_i&=|\{c\in V^*\mid \wt(c|_T)=i\}|.
\end{align*}
Then
\begin{align*}
b_0&=|\{c\in V^*\mid c|_T=0\}|
\nexteq
|u_1^\perp\cap u_2^\perp\cap u_3^\perp|
\nexteq
\begin{cases}
2^{m-3}&\text{ if $u_1+u_2\neq u_3$,}\\
2^{m-2}&\text{otherwise,}
\end{cases}
\displaybreak[0]\\
b_1&=|\{c\in V^*\mid \wt(c|_T)=1\}|
\nexteq
3|\{c\in V^*\mid c(u_1)=1,\;c(u_2)=c(u_3)=0\}|
\nexteq
3|u_2^\perp\cap u_3^\perp\setminus u_1^\perp|
\nexteq
\begin{cases}
3(|u_2^\perp\cap u_3^\perp|-|u_2^\perp\cap u_3^\perp\cap u_1^\perp|)
&\text{if $u_1+u_2\neq u_3$,}\\
0&\text{otherwise,}
\end{cases}
\nexteq
\begin{cases}
3\cdot2^{m-3}
&\text{if $u_1+u_2\neq u_3$,}\\
0&\text{otherwise,}
\end{cases}
\displaybreak[0]\\
b_2&=|\{c\in V^*\mid \wt(c|_T)=2\}|
\nexteq
3|\{c\in V^*\mid c(u_1)=c(u_2)=1,\;c(u_3)=0\}|
\nexteq
3|u_3^\perp\setminus (u_1^\perp\cup u_2^\perp)|
\nexteq
3(|u_3^\perp|-|u_3^\perp\cap u_1^\perp|-|u_3^\perp\cap u_2^\perp|+
|u_3^\perp\cap u_1^\perp\cap u_2^\perp|)
\nexteq
\begin{cases}
3\cdot2^{m-3}&\text{if $u_1+u_2\neq u_3$,}\\
3\cdot2^{m-2}&\text{otherwise,}
\end{cases}
\displaybreak[0]\\
b_4&=0,\\
b_3&=|V^*|-b_0-b_1-b_2-b_4
\nexteq
2^m-b_0-b_1-b_2
\nexteq
\begin{cases}
2^m-7\cdot2^{m-3}&\text{if $u_1+u_2\neq u_3$,}\\
2^m-4\cdot2^{m-2}&\text{otherwise,}
\end{cases}
\nexteq
\begin{cases}
2^{m-3}&\text{if $u_1+u_2\neq u_3$,}\\
0&\text{otherwise.}
\end{cases}
\end{align*}
Since
\begin{align*}
a_i&=|\{c\in (V^*\setminus\{0\})\cup((V^*+\allone)\setminus\{\allone\})
\mid \wt(c|_T)=i\}|
\nexteq
|\{c\in V^*\setminus\{0\}\mid \wt(c|_T)=i\}|+
|\{c\in (V^*+\allone)\setminus\{\allone\}\mid \wt(c|_T)=i\}|
\nexteq
b_i-\delta_{i,0}
+|\{c\in V^*\setminus\{0\}\mid \wt(c|_T)=4-i\}|
\nexteq
b_i-\delta_{i,0}+b_{4-i}-\delta_{i,4},
\end{align*}
we obtain the desired results.
\end{proof}

\begin{proof}[Proof of Theorem \ref{thm:Jacobi}]
Part (\ref{thm:1.1-1}) follows from
Lemma~\ref{lem:1}, by noticing that
$\wt(c)=2^{m-1}$ for all $c\in C\setminus\{0,\allone\}$.
Part (\ref{thm:1.1-2}) follows from part (\ref{thm:1.1-1})
and Theorem~\ref{thm:Mac-Jacobi}.
\end{proof}



\section{Proof of Theorem \ref{thm:main1}}\label{sec:harm}

In this section, 
we prove Theorem \ref{thm:main1}.
Let $V=\FF_2^m$. 

\begin{lem}\label{lem:numwords}
Let $C=RM(1,m)$ and 
$U\subset V$ be a three-dimensional subspace of $V$. 
Then $\{c\vert_U\mid c\in C\}=H_8$, 
and for each 
$a\in H_8$, we have
\[|\{c\in C_{2^{m-1}}\mid c|_U=a\}|=
2^{m-3}. \]
\end{lem}
\begin{proof}
The first statement is immediate since $RM(1,3)=H_8$. 
As for the second statement, we may assume $a$ has entry $0$
at the coordinate $0\in V$. Then the number of $c\in C_{2^{m-1}}$
with $c|_U=a$ is the same as 
the number of $\lambda\in V^\ast$ satisfying 
$\kernel\lambda\supset U$, which is $2^{m-3}$.
\end{proof}


\begin{proof}[Proof of Theorem \ref{thm:main1}]
Since $f$ is a harmonic function of degree $k\geq1$,
we have $\widetilde{f}(\emptyset)=0$. 
Since $\gamma (f)=0$ and 
\begin{align*}
\sum_{y\in\binom{U}{k-1}}(\gamma (f))(y)
=k\widetilde{f}(U),
\end{align*}
we have
$\widetilde{f}(U)=0$.
Note that
$\{c\vert_U\mid c\in C\}=H_8$ has codewords of weight $0,4,8$ only.
This implies that,
for $c\in C$,
$\widetilde{f}(c)=\widetilde{f}(c|_{U})$
is nonzero only if $\wt(c|_U)=4$, or equivalently,
$\wt(c)=2^{m-1}$.
Then 
\begin{align*}
w_{C,f}(x,y)&=
\sum_{c\in C,\;\wt(c)=2^{m-1}}
\widetilde{f}(c\vert_U)x^{2^m-\wt(c)}y^{\wt(c)}
\nexteq
x^{2^{m-1}}y^{2^{m-1}}\sum_{a\in H_8, \wt(a)=4}|\{c\in C_{2^{m-1}}
 \mid c\vert_U=a\}|\widetilde{f}(a) 
\nexteq
2^{m-3}x^{2^{m-1}}y^{2^{m-1}}\sum_{a\in H_8, \wt(a)=4}\widetilde{f}(a)
\end{align*}
by Lemma \ref{lem:numwords}. 
Then we obtain (1). 
Using Theorem \ref{thm: Bachoc iden.}, we obtain (2). 
\end{proof}
\section{Proof of Corollary \ref{cor:main}}\label{sec:cor}

In this section, 
we give a proof of Corollary \ref{cor:main}. 

\begin{proof}[Proof of Corollary \ref{cor:main}]
Let $C=RM(1,m)$. 
We give two proofs. 

The first proof relies on properties of Jacobi polynomials. 
Let 
$T_1=\{0,u_1,u_2,u_3\}\in\binom{V}{4}$ and 
$T_2=\{0,v_1,v_2,v_3\}\in\binom{V}{4}$.
We assume that 
$u_1+u_2\neq u_3$ and 
$v_1+v_2=v_3$. 
By Theorem \ref{thm:Jacobi}, 
\[
J_{C,T_1}-J_{C,T_2}=-2^{m-3} x^{2^{m-1}-4} y^{2^{m-1}-4} 
(w y - x z)^4. 
\]
The coefficient of $z^{4} x^{2^m-\ell} y^{\ell-4}$ in 
$J_{C,T_1}-J_{C,T_2}$ is non-zero whenever $C_{\ell}$ is non-empty. 
Hence, $C_{\ell}$ is not a $4$-design. 

By using Theorem \ref{thm:Mac-Jacobi}, we have
\[
J_{C^\perp,T_1}-J_{C^\perp,T_2}=-(x^2-y^2)^{2^{m-1}-4}(w y - x z)^4. 
\]
The coefficient of $z^{4} x^{2^m-\ell} y^{\ell-4}$ in 
$J_{C^\perp,T_1}-J_{C^\perp,T_2}$ is non-zero 
whenever $(C^\perp)_{\ell}$ is non-empty. 
Hence, $(C^\perp)_{\ell}$ is not a $4$-design.




The second proof relies on properties of harmonic weight enumerators. 
Let $U\subset V$ be a three-dimensional subspace of $V$. 
Then 
$\{c|_U\mid c\in C\}=H_8$ by Lemma~\ref{lem:numwords}.
Let $\mathcal{B}=\{\supp(z)\mid z\in H_8,\; \wt(z)=4\}$, and 
let $\tau$ be a transposition of two coordinates of $H_8$.
Then $\mathcal{B}$ is the set of $14$ blocks of a $3$-$(8,4,1)$ design, 
and $|\mathcal{B}\cap\mathcal{B}^\tau|=6$.
Define 
\[f=\sum_{z\in\mathcal{B}} z-\sum_{z\in\mathcal{B^\tau}} z.\]
Then $f\in\Harm_4$, and
\begin{align*}
\sum_{\substack{a\in H_8\\ \wt(a)=4}} f(a)
&=|\mathcal{B}|-|\mathcal{B}\cap\mathcal{B}^\tau|=
8.
\end{align*}
By Theorem \ref{thm:main1} we have 
\begin{align*}
w_{C,f}(x,y)&=8\cdot 2^{m-3}x^{2^{m-1}}y^{2^{m-1}},\\
w_{C^\perp,f}(x,y)&=8
(xy)^4(x^2-y^2)^{2^{m-1}-4}, 
\end{align*}
and the coefficients of $x^{2^m-\ell}y^\ell$ in $w_{C,f}$ 
(resp.\ $w_{C^\perp,f}$) do not vanish 
whenever
$C_{\ell}$ (resp.\ $(C^\perp)_\ell$) is non-empty. 
By Theorem \ref{thm:design}, 
the proof is complete. 
\end{proof}

\section*{Acknowledgments}
The authors are supported by JSPS KAKENHI (20K03527, 22K03277).

\section*{Statements and Declarations}
Competing Interests. The authors have no affiliation with any organization 
with a direct or indirect financial interest in the subject 
matter discussed in the manuscript. 



\begin{thebibliography}{999}

\bibitem{assmus-mattson}
E.F.~Assmus,~Jr. and  H.F.~Mattson,~Jr., 
New $5$-designs, 
\emph{J. Combin. Theory} {\bf 6} (1969), 122--151.

\bibitem{Bachoc}
C.~Bachoc,
On harmonic weight enumerators of binary codes,
\emph{Des. Codes Cryptogr.} {\bf 18} (1999), no.~1-3, 11--28. 




\bibitem{BS}
A.~Bonnecaze and P.~Sole, 
The extended binary quadratic residue code of length 42 holds a 3-design,
\emph{J.~Combin.~Des.} {\bf 29} (2021), no.~8, 528--532. 

\bibitem{Magma}
W.~Bosma, J.~Cannon, and C.~Playoust.
The Magma algebra system.~I.\ The user language, 
\emph{J. Symb. Comp.} {\bf 24} (1997), 235--265.






\bibitem{Delsarte}
P.~Delsarte, 
{Hahn polynomials, discrete harmonics, and $t$-designs}, 
\emph{SIAM J. Appl. Math.} {\bf 34} (1978), no.~1, 157--166. 


\bibitem{Dillon-Schatz}
J.F.~Dillon and J.R.~Schatz, 
``Block designs with the symmetric difference property'',
in: Proc. of the NSA Mathematical Sciences Meetings, (Ward R. L. Ed.),
pp. 159--164, 1987.
















\bibitem{MS}
F.J.~MacWilliams and N.J.A.~Sloane, 
{\sl The theory of error-correcting codes.~II.} 
North-Holland Mathematical Library, Vol.~16.~
North-Holland Publishing Co., Amsterdam-New York-Oxford, 1977. 
pp.~i-ix and 370--762. 




\bibitem{MMN}
T.~Miezaki, A.~Munemasa, and H.~Nakasora,
A note on Assmus--Mattson type theorems,
\emph{Des. Codes Cryptogr.}, 
{\bf 89} (2021), 843--858.


\bibitem{MN-TEC}
T.~Miezaki and H.~Nakasora,
The support designs of the triply even binary codes of length $48$,
\emph{J. Combin. Designs}, 
{\bf 27} (2019), 673--681.

\bibitem{mn-typeI}
T.~Miezaki and H.~Nakasora,
On the Assmus--Mattson type theorem 
for Type I and even formally self-dual codes, 
\emph{J. Combin. Designs}, 
{\bf 31}, (2023) no.~7, 335--344. 

\bibitem{Ozeki}
M.~Ozeki, 
On the notion of Jacobi polynomials for codes. 
\emph {Math.~Proc.~Cambridge Philos.~Soc.} {\bf 121} (1997), no.~1, 15--30. 







\bibitem{Mathematica}
Wolfram Research, Inc., Mathematica, Version 11.2, Champaign, IL (2017).


\end{thebibliography}
\end{document}